\documentclass[12pt,reqno]{amsart}
\usepackage{amsmath}
\usepackage{amssymb}
\usepackage{amstext}
\usepackage{mathrsfs}
\usepackage{a4wide}
\usepackage{graphicx}
\usepackage{bbm}
\usepackage{verbatim}
\usepackage{bm}
\usepackage[colorlinks, linkcolor=black, anchorcolor=black,
citecolor=black]{hyperref}
\usepackage{hyperref}
\hypersetup{hypertex=true,
            colorlinks=true,
            linkcolor=blue,
            anchorcolor=blue,
            citecolor=blue }

\allowdisplaybreaks \numberwithin{equation}{section}
\usepackage{color}
\usepackage{cases}

\numberwithin{equation}{section}

\newtheorem{theorem}{Theorem}[section]
\newtheorem{proposition}[theorem]{Proposition}

\theoremstyle{definition}

\theoremstyle{remark}
\newtheorem{remark}[theorem]{Remark}
\newtheorem{conj}[theorem]{Conjecture}

\begin{document}
\title[A Liouville theorem for the Euler equations in a disk]{A Liouville theorem for the Euler equations in a disk}
	
	\author{Yuchen Wang, Weicheng Zhan}
	
\address{School of Mathematical Science
Tianjin Normal University, Tianjin, 300074,
P.R. China}
\email{wangyuchen@mail.nankai.edu.cn}

\address{School of Mathematical Sciences, Xiamen University, Xiamen, Fujian, 361005, P.R. China}
\email{zhanweicheng@amss.ac.cn}


	\begin{abstract}
We present a symmetry result regarding stationary solutions of the 2D Euler equations in a disk. We prove that in a disk, a steady flow with only one stagnation point and tangential boundary conditions is a circular flow, which confirms a conjecture proposed by F. Hamel and N. Nadirashvili in [J. Eur. Math. Soc., 25 (2023), no. 1, 323-368]. The key ingredient of the proof is to use `local' symmetry properties for the non-negative solutions of semi-linear elliptic equations with a continuous nonlinearity in a ball, which can be established by a rearrangement technique called continuous Steiner symmetrization.
	\end{abstract}
	
	\maketitle{\small{\bf Keywords:} The Euler equation, Circular flows, Liouville theorem, Semilinear elliptic equation. \\

	
	\section{Introduction and Main result}
We are concerned in this note with stationary solutions of the 2D Euler equations in a planar domain $\Omega\subset \mathbb{R}^2$
	\begin{align}\label{1-1}
		\begin{cases}
			\mathbf{v}\cdot \nabla \mathbf{v} =-\nabla P&\text{in}\ \ \Omega,\\
			\nabla\cdot\mathbf{v}=0\,\ \, \ \ \ \ \ \ \  \ \, &\text{in}\ \ \Omega,
\end{cases}
	\end{align}
where $\mathbf{v}=(v_1, v_2)$ is the velocity field and $P$ is the scalar pressure. Here the solutions $\mathbf{v}$ and $P$ are always understood in the classical sense, that is, they are (at least) of class $C^1$ in $\Omega$, and therefore satisfy \eqref{1-1} everywhere in $\Omega$. When the boundary $\partial \Omega$ is not empty, we also need to assume some appropriate boundary conditions.

We are interested in investigating how the geometry of domain $\Omega$ affects the properties of steady flows, specifically identifying conditions (as weak as possible) that ensure that solutions inherit the geometric symmetry properties of the domain. For example, a famous result by Hamel and Nadirashvili \cite{HN3} shows that in a strip, a steady flow with no stationary point and tangential boundary conditions is a shear flow.
We refer the reader to \cite{ Chae, Cons, Gom, HN3, HN2, HN1, HN, Nad, Ruiz} for some relevant results in this aspect.

Let $D$ be an open non-empty disk with radius $R>0$ centered at the origin, and ${n}$ be the outward unit normal on $\partial D$. We say that a flow $\mathbf{v}$ is a circular flow if $\mathbf{v}(x)$ is parallel to the vector $\mathbf{e}_\theta(x)=\left(-x_2/|x|, x_1/|x|\right)$ at every point $x\in D\backslash\{0\}$. Recently, Hamel and Nadirashvili \cite{HN} proposed the following conjecture:
\begin{conj}[\cite{HN}, Conjecture 1.12]\label{con1}
  Let $z\in D$ and let $\mathbf{v}$ be a $C^2(\overline{D}\backslash \{z\})$ and bounded flow solving \eqref{1-1} with $\Omega=D\backslash\{z\}$ and $\mathbf{v}\cdot n=0$ on $\partial D$. Assume that $|\mathbf{v}|>0$ in $\overline{D}\backslash\{z\}$. Then $z$ is the origin and $\mathbf{v}$ is a circular flow.
\end{conj}

 We would like to point out that although this conjecture is quite natural, it seems hard to give a rigorous proof; see pages 332-333 in \cite{HN} for some discussions of  main difficulties. A related weaker conjecture is more promising, which is also stated in \cite{HN}.
 \begin{conj}[\cite{HN}, a weaker version of Conjecture \ref{con1}]\label{con2}
    Let $\mathbf{v}\in C^2(\overline{D})$ solve \eqref{1-1} with $\Omega=D$ and $\mathbf{v}\cdot n=0$ on $\partial D$. Assume that $z\in D$ is the only stagnation point of $\mathbf{v}$ in $\overline{D}$, that is, $|\mathbf{v}(z)|=0$ and $|\mathbf{v}|>0$ in $\overline{D}\backslash\{z\}$. Then $z$ is the origin and $\mathbf{v}$ is a circular flow.
 \end{conj}

The purpose of this note is to give a positive answer to Conjecture \ref{con2}. Indeed, we prove a slightly stronger result which only requires the flow to have a single stagnation point within the domain.
\begin{theorem}\label{th}
   Let $\mathbf{v}\in C^2(\overline{D})$ solve \eqref{1-1} with $\Omega=D$ and $\mathbf{v}\cdot n=0$ on $\partial D$. Assume that $z\in D$ is the only stagnation point of $\mathbf{v}$ in $D$, that is, $|\mathbf{v}(z)|=0$ and $|\mathbf{v}|>0$ in $D\backslash\{z\}$. Then $z$ is the origin and $\mathbf{v}$ is a circular flow.
\end{theorem}

The proof of Theorem \ref{th} is provided in Section 2.

\section{Proof of Theorem \ref{th}}
In this section, we give a proof to Theorem \ref{th}. We will follow the strategy in \cite{HN} and turn the problem into the symmetry of solutions of a semilinear elliptic boundary value problem in $D$.

The flow $\mathbf{v}$ has a stream function $u:\overline{D}\to \mathbb{R}$ of class $C^3(\overline{D})$ defined by
\begin{equation*}
  \nabla^\perp u=\mathbf{v},\ \ \ \text{that is}\ \ \ \partial_1 u=v_2\ \ \text{and}\ \ \partial_2u=-v_1
\end{equation*}
in $\overline{D}$, since $D$ is simply connected and $\mathbf{v}$ is divergence free. The tangency condition $\mathbf{v}\cdot {n}=0$ on $\partial D$ implies that $u$ is constant along $\partial D$. Such stream function $u$ is uniquely defined in $\overline{D}$ up to an additive constant. Up to normalization, we may assume, without loss of generality, that
\begin{equation*}
  u=0\ \ \ \text{on}\ \ \partial D.
\end{equation*}
Since $z\in D$ is the only stagnation point of $\mathbf{v}$ in $D$, the stream function $u$ has a unique critical point in $D$. We may assume, without loss of generality, that $u$ has a unique maximum point in $\overline{D}$ and this point is actually the stagnation point $z$ (after possibly changing $\mathbf{v}$ into $-\mathbf{v}$ and $u$ into $-u$).
 The uniqueness of the critical point of $u$ in $D$ implies that
\begin{equation*}
  0<u(x)<u(z)\ \ \ \text{for all}\ \ x\in D\backslash\{z\}.
\end{equation*}
By Proposition \ref{p-1} in Section 3, we see that there is a continuous function $f: [0, u(0)] \to \mathbb{R}$ such that
  \begin{equation*}
   \Delta u+f(u)=0\ \ \ \text{in}\ \ D.
  \end{equation*}
It remains to show that $u$ is a radially decreasing function with respect to the origin. Indeed, from Proposition \ref{p-2} in Section 3, one knows that $u$ is locally symmetric, namely, it is  radially symmetric and radially decreasing in some annuli (probably infinitely many) and flat elsewhere. Recall that $u$ has a unique critical point in $D$. We conclude that the number of annuli can only be one at most, and hence $u$ is a radially decreasing function. The proof is thus complete.

\section{Auxiliary results}
In this section, we collect two auxiliary results, which have been used in the proof.

The first result states that the corresponding stream function of a steady flow satisfies some semi-linear elliptic equation under certain extra assumptions.
\begin{proposition}\label{p-1}
  Let $\mathbf{v}$ be as in Theorem \ref{th} and let $u\in C^3(\overline{D})$ be the corresponding stream function, and let $J$ its range defined by
  \begin{equation*}
    J=\{u(x)\mid x\in \overline{D}\}.
  \end{equation*}
  Then there is a continuous function $f: J \to \mathbb{R}$ such that
  \begin{equation*}
   \Delta u+f(u)=0\ \ \ \text{in}\ \ \overline{D}.
  \end{equation*}
\end{proposition}

\begin{proof}
The result can be proved by using a similar argument as in \cite{HN}. For the reader's convenience, we present the detailed proof here. Without loss of generality, we may assume that
\begin{equation*}
  u=0\ \,\text{on}\ \,\partial D\ \ \ \text{and}\ \ \ 0<u<u(z)\ \,\text{in}\ \,D\backslash\{z\}.
\end{equation*}
Consider any point $y\in D\backslash\{z\}$. Let $\sigma_y$ be the solution of
	\begin{align}\label{3-10}
		\begin{cases}
			 \dot{\sigma}_y(t) = \nabla u(\sigma_y(t)), &\\
             \sigma_y(0)=y.\ \ &
\end{cases}
	\end{align}
Then by Lemma 2.2 in \cite{HN}, there are some quantities $t_y^{\pm}$ such that $-\infty\le t^-_y<0<t^+_y\le +\infty$ and the solution $\sigma_y$ of \eqref{3-10} is of class $C^1((t_y^-, t^+_y))$ and ranges in $D\backslash \{z\}$, with
	\begin{align}\label{3-11}
		\begin{cases}
		|\sigma_y(t)|\to R\ \,\text{and}\ \,u(\sigma_y(t))\to 0\ \,\text{as}\ \,t\to t_y^-,\ \ &\\
|\sigma_y(t)-z|\to 0 \ \,\text{and}\ \,u(\sigma_y(t))\to u(z)\ \,\text{as}\ \, t\to t_y^+. &
\end{cases}
	\end{align}
Set $g:=u\circ \sigma_y\in C^1((t_y^-, t_y^+))$. Then $g$ is increasing since $(u\circ \sigma_y)'(t)=|\nabla u(\sigma_y(t))|^2=|\mathbf{v}(\sigma_y(t))|^2>0$ for all $t\in (t_y^-, t_y^+)$. So $g$ is an increasing homeomorphism from $(t_y^-, t_y^+)$ onto $(0, u(z))$. Consider the function $f: (0, u(z))\to \mathbb{R}$ defined by
\begin{equation}\label{3-12}
  f(\tau)=-\Delta u(\sigma_y(g^{-1}(\tau)))\ \ \ \text{for}\ \ \tau\in (0, u(z)).
\end{equation}
Then $f$ is of class $C^1((0, u(z)))$ by the chain rule. The equation $\Delta u+f(u)=0$ is now satisfied along the curve $\sigma_y((t_y^-, t_y^+))$. Let us check it in the whole set $D$. Consider first any point $x\in D$. Let $\xi_x$ be the solution of
	\begin{align*}
		\begin{cases}
			 \dot{\xi}_x(t) = \mathbf{v}(\xi_x(t)), &\\
             \xi_x(0)=x.\ \ &
\end{cases}
	\end{align*}
Then $\xi_x$ is defined in $\mathbb{R}$ and periodic. Furthermore, the streamline $\Phi_x:=\xi_x(\mathbb{R})$ is a $C^1$ Jordan curve surrounding $z$ in $D$ and meets the curve $\sigma_y((t_y^-, t_y^+))$ once; see Lemma 2.6 in \cite{HN}. Hence, there is $s\in (t_y^-, t_y^+)$ such that $\sigma_y(s)\in \Phi_x$. Note that both the stream function $u$ and the vorticity $\Delta u$ are constant along the streamline $\Phi_x$. It follows from \eqref{3-12} that
\begin{equation*}
  \Delta u(x)+f(u(x))=\Delta u(\sigma_y(s))+f(u(\sigma_y(s)))=\Delta u(\sigma_y(s))+f(g(s))=0.
\end{equation*}
Therefore, $\Delta u+f(u)=0$ in $D$. By the continuity of $u$ in $\overline{D}$, we have
\begin{equation*}
 \min_{t\in \mathbb{R}}|\xi_x(t)|\to R  \ \text{as}\ |x|\to R\ \ \ \text{and}\ \ \   \max_{t\in \mathbb{R}}|\xi_x(t)-z|\to 0\ \text{as}\ |x-z|\to 0.
\end{equation*}
Since $\Delta u$ is uniformaly continuous in $\overline{D}$ and constant along any streamline of the flow, we see that $\Delta u$ is constant on $\partial D$. Call $d$ the value of $\Delta u$ on $\partial D$. Set $f(0)=-d$ and $f(u(z))=-\Delta u (z)$. Then we infers from \eqref{3-11} and \eqref{3-12} that $f:[0, u(z)]\to \mathbb{R}$ is continuous in $[0, u(z)]$ and that the equation $\Delta u+f(u)=0$ holds in $\overline{D}$. The proof is thus complete.
\end{proof}

The following result is about the symmetry of solutions to semi-linear elliptic equations with a continuous nonlinearity in a ball, which can be found in \cite{Bro1}(see also \cite{Bro2}). Such symmetry results are obtained by a rearrangement technique called continuous Steiner symmetrization; see \cite{Bro0,Bro1}.

\begin{proposition}[\cite{Bro1}, Theorem 7.2]\label{p-2}
  Let $B_R$ be a ball in $\mathbb{R}^N$, with radius $R>0$ centered at the origin, and let $g=g(r,t)$ be of class $ C([0, R]\times [0, +\infty))$ and be non-increasing in $r$. Let $u\in C^1(\overline{B_R})$ be a weak solution of the following problem
	\begin{align*}
		\begin{cases}
			-\Delta u=f(|x|, u),\ \ u>0&\text{in}\ \ B_R,\\
             u=0,\ \ &\text{on}\  \ \partial B_R.
\end{cases}
	\end{align*}
  The $u$ is locally symmetric in the following sense:
  \begin{itemize}
    \item [(1)]$\displaystyle  B_R=\bigcup_{k=1}^m A_k \cup \{x\mid \nabla u(x)=0\}$, \ \ \ \text{where}
    \begin{equation*}
      A_k=B_{R_k}(z_k)\backslash \overline{B_{r_k}(z_k)},\ \ \ z_k\in B_R,\ \ \ 0\le r_k<R_k;
    \end{equation*}
    \smallskip
    \item [(2)]$u(x)=U_k(|x-z_k|)$, $x\in A_k$, where $U_k\in C^1([r_k, R_k])$;
    \smallskip
    \item [(3)]$U'_k(r)<0$ for $r\in (r_k, R_k)$;
        \smallskip
    \item [(4)]$u(x)\ge U_k(r_k),\ \forall\,x\in B_{r_k}(z_k)$, $k=1,\cdots, m$;
        \smallskip
    \item [(5)] the sets $A_k$ are pairwise disjoint and $m\in \mathbb{N}\cup \{+\infty\}$.
  \end{itemize}
\end{proposition}
\begin{remark}
  Note that if $u$ is locally symmetric, then $u$ is radially symmetric and radially decreasing in annuli $A_k\, (k=1, \cdots, m)$, and flat elsewhere in $B_R$. Moreover, since $u\in C^1(\overline{B_R})$, we have that
  \begin{equation*}
    U_k'(r_k)=0,
  \end{equation*}
  and if $R_k<R$, then also
 \begin{equation*}
    U_k'(R_k)=0,
 \end{equation*}
 $(k\in \{1,\cdots, m\})$.
\end{remark}

\noindent{\bf Acknowledgments.}
{The authors are grateful to Professor François Hamel for his helpful discussion on this issue.
}

	\phantom{s}
	\thispagestyle{empty}


\begin{thebibliography}{99}

\bibitem{Bro0}
 F. Brock, Continuous Steiner-symmetrization, \textit{Math. Nachr.}, 172 (1995), 25--48.

\bibitem{Bro1}
 F. Brock, Continuous rearrangement and symmetry of solutions of elliptic problems, \textit{Proc. Indian Acad. Sci. Math. Sci.}, 110 (2000), no. 2, 157--204.

\bibitem{Bro2}
F. Brock and P. Takáč, Symmetry and stability of non-negative solutions to degenerate elliptic equations in a ball, \textit{Proc. Amer. Math. Soc.}, 150 (2022), no. 4, 1559--1575.

\bibitem{Chae}
D. Chae and P. Constantin, Remarks on a Liouville-type theorem for Beltrami flows, \textit{Int. Math. Res. Not. IMRN}, 2015, no. 20, 10012--10016.

\bibitem{Cons}
P. Constantin, T. D. Drivas and D. Ginsberg, Flexibility and rigidity in steady fluid motion, \textit{Comm. Math. Phys.}, 385 (2021), no. 1, 521--563.

\bibitem{Gom}
J. Gómez-Serrano, J. Park, J. Shi and Y. Yao, Symmetry in stationary and uniformly rotating solutions of active scalar equations, \textit{Duke Math. J.}, 170 (2021), no. 13, 2957--3038.



    \bibitem{HN3}
F. Hamel and N. Nadirashvili, Shear flows of an ideal fluid and elliptic equations in unbounded domains. Comm. Pure Appl. Math. 70 (2017), no. 3, 590–608.

    \bibitem{HN2}
F. Hamel and N. Nadirashvili, Parallel and circular flows for the two-dimensional Euler equations. Semin. Laurent Schwartz EDP Appl. 2017–2018, exp. V, 1--13.

    \bibitem{HN1}
    F. Hamel and N. Nadirashvili, A Liouville theorem for the Euler equations in the plane, \textit{Arch. Ration. Mech. Anal.}, 233 (2019), no. 2, 599--642.



    \bibitem{HN}
    F. Hamel and N. Nadirashvili, Circular flows for the Euler equations in two-dimensional annular domains, and related free boundary problems, \textit{J. Eur. Math. Soc. (JEMS)}, 25 (2023), no. 1, 323--368.

\bibitem{Nad}
N. Nadirashvili, Liouville theorem for Beltrami flow, \textit{Geom. Funct. Anal.} 24 (2014), no. 3, 916--921.

\bibitem{Ruiz}
D. Ruiz, Symmetry results for compactly supported steady solutions of the 2D Euler equations, \textit{Arch. Ration. Mech. Anal.}, 247 (2023), no. 3, Paper No. 40, 25 pp.

	

	\end{thebibliography}
\end{document}